\documentclass[11pt]{article}      

\usepackage{amssymb}      
\usepackage{amsmath}      
\usepackage{amsthm}      
\usepackage{amsbsy}      
\usepackage{amscd}      
\usepackage[dvips]{graphicx}

\theoremstyle{plain}      
\newtheorem{theorem}{Theorem}[section]      
\newtheorem{lemma}{Lemma}[section]      
\newtheorem{corollary}[theorem]{Corollary}      
\newtheorem{proposition}{Proposition}[section]      
\newtheorem{conjecture}{Conjecture}[section]      
      
\newtheorem{definition}{Definition}[section]          
\theoremstyle{remark}      
\newtheorem{remark}{Remark}[section]

\newcommand{\Z}{{\mathbb{Z}}}   
   
\newcommand{\C}{{\mathbb{C}}}      
\newcommand{\R}{{\mathbb{R}}}      

\begin{document}

\date{\today}

\title{On Burau representations at roots of unity}         
\author{      
\begin{tabular}{cc}      
 Louis Funar &  Toshitake Kohno\\      
\small \em Institut Fourier BP 74, UMR 5582       
&\small \em IPMU, Graduate School of Mathematical Sciences\\      
\small \em University of Grenoble I &\small \em The University of Tokyo    \\      
\small \em 38402 Saint-Martin-d'H\`eres cedex, France      
&\small \em 3-8-1 Komaba, Meguro-Ku, Tokyo 153-8914 Japan \\      
\small \em e-mail: {\tt funar@fourier.ujf-grenoble.fr}      
& \small \em e-mail: {\tt kohno@ms.u-tokyo.ac.jp} \\      
\end{tabular}      
}

\maketitle 

\begin{abstract}
We consider subgroups of the braid groups which are 
generated by $k$-th powers of the standard generators 
and prove that any infinite intersection (with   
even $k$) is trivial. This is motivated by some conjectures of 
Squier concerning the kernels of Burau's representations of the braid groups at 
roots of unity. 
Furthermore, we show that the image of the braid group 
on 3 strands by these representations is either a finite group, 
for a few roots of unity, or a finite extension of a triangle group, by 
using geometric methods.

\vspace{0.1cm}
 

\end{abstract}

\section{Introduction and statements}

The first part of the present paper is devoted to the 
study of  groups related to the kernels of  
Burau's  representations of the braid groups at roots of unity. 
We consider two conjectures stated by Squier in \cite{Sq} 
concerning these kernels. These conjectures were part of 
an approach to the faithfulness of Burau's representations 
and it seems that they were overlooked over the years because of the counterexamples 
found by Moody, Long, Paton and Bigelow (see \cite{Moody,LP,Big}) 
for braids on $n\geq 5$ strands. 

\vspace{0.2cm}\noindent
Specifically, let $B_n$ denote the braid group on $n$ strands with 
the standard generators $g_1,g_2,\ldots,g_{n-1}$. Squier was interested 
to compare the kernel of Burau's representation $\beta_{q}$ at a 
$k$-th root of unity $q$ with the normal subgroup $B_n[k]$ 
of $B_n$ generated by $g_j^{k}$, $1\leq j\leq n-1$. 
Our first result answers a strengthened form of the conjecture C2 in \cite {Sq}:

\begin{theorem}\label{intersection}
The intersection of $B_n[2k]$ over any infinite set 
of integers $k$ is trivial. 
\end{theorem}

\vspace{0.2cm}\noindent
Our method does not give any information about the intersection of 
$B_n[k]$ with odd $k$.  

\vspace{0.2cm}\noindent
The proof  uses the asymptotic faithfulness 
of quantum representations of mapping class groups, 
due to Andersen (\cite{A}) and independently 
to Freedman, Walker and Wang (\cite{FWW}). 
The other conjecture stated in \cite{Sq} is that $B_n[k]$ 
is the kernel of Burau's representation. This is  false 
because Burau's representation at a generic parameter 
is not faithful for $n\geq 5$ (see Proposition \ref{false}).

\vspace{0.2cm}\noindent
The main body of the paper is devoted to the complete 
description of the image of Burau's representation of $B_3$.
We can state our main result in this direction as follows:  

\begin{theorem}\label{B3}
Assume that $q$ is a primitive $n$-th root of unity and $g_1,g_2$ are the 
standard generators of $B_3$. Then 
$\beta_{-q}(B_3)$ has a presentation with generators 
$g_1,g_2$ and relations: 
\begin{enumerate}
\item The case $n=2k$ and $k$ is odd: 
\[\begin{array}{lll}
\mbox{\rm Braid relation:} &  &  g_1g_2g_1=g_2g_1g_2, \\
\mbox{\rm Power relations:} &  & g_1^{2k}=g_2^{2k}=(g_1^2g_2^2)^k=1. \\ 
\end{array}\]
\item The case  $n= 2k$ and $k$ is even: 
\[\begin{array}{lll}
\mbox{\rm Braid relation:} &  &  g_1g_2g_1=g_2g_1g_2, \\
\mbox{\rm Power relations:} &  & g_1^{2k}=g_2^{2k}=(g_1^2g_2^2)^{2k}=1. \\
\end{array}\]
\item The case $n= 2k+1$: 
\[\begin{array}{lll}
\mbox{\rm Braid relation:} &  &   g_1g_2g_1=g_2g_1g_2, \\
\mbox{\rm Power relations:} &  &  g_1^{2k+1}=g_2^{2k+1}=(g_1^2g_2^2)^{2(2k+1)}=1.\\
\end{array}\]
\end{enumerate}
\end{theorem}

\vspace{0.2cm}\noindent
A similar result was obtained independently by Masbaum in \cite{Mas1} in a 
slightly different context. Consider the 
2-dimensional $SO(3)$-quantum representations of the mapping class group 
of the punctured torus at a  primitive $2p$-th root 
of unity for odd $p$, with the puncture labeled by 
the color $c=\frac{p-1}{2}-2$. Then the result proved by 
Masbaum is that the kernel of this representation  
is normally generated by the $p$-th powers of the Dehn twists. 
However, these  quantum representations are covered by 
Burau's representations of $B_3$, so the two results above are 
equivalent. The same arguments apply to the 
quantum representations of  the mapping class group $M_{0,4}$ of the 
$4$-holed sphere. Notice that 2-dimensional representations 
of $B_3$ are equivalent either to abelian representations, 
to some not completely reducible representations, or else 
to Burau's representation.  
Another consequence of this theorem is the fact 
that the image of a pseudo-Anosov mapping class in the 
mapping class group of the punctured (or holed) torus  
by the quantum representations considered above is of infinite order 
for $p$ large enough.  This  solves a particular case of a conjecture 
formulated by Andersen, Masbaum and Ueno in \cite{AMU}; 
a proof of the conjecture in
this case was announced by Masbaum in \cite{AMU}, Remark 5.9
(see also \cite{GM2}, p.4).

\vspace{0.2cm}\noindent
The proof of this algebraic statement has a strong geometric flavor. 
A key ingredient is Squier's theorem concerning the 
unitarizability of Burau's representation (see \cite{Sq}). 
The non-degenerate Hermitian form defined by Squier is 
invariant under the braid group, but it is not always positive definite. 
First, we find whether it is positive definite, so that the representation can 
be conjugate into $U(2)$.  On the other hand, when this Hermitian form is 
not positive definite, the representation can be 
complex-unitarized, namely it can be (rescaled and) conjugated into $U(1,1)$.

\vspace{0.2cm}\noindent
We will then focus then on the complex-unitary case.  
We show that that the image of some free subgroup 
of  the pure braid group $PB_3$ on three strands  
by Burau's representation is a subgroup of $PU(1,1)$ 
generated by three rotations in the hyperbolic plane. 
Here, the hyperbolic plane is identified to the 
unit disk of  the  complex projective  line $\C P^1$. 
Geometric arguments due to Knapp, Mostow and Deraux (see 
\cite{Knapp,Mo,Deraux}) show that the image of $PB_3$ is a discrete 
triangle group and thus we can give an explicit presentation for it. 
Then an easy argument permits to describe the image of 
the slightly larger group $B_3$. In particular, we obtain a 
description of the kernel of Burau's representation of 
$B_3$ at roots of unity, which will give a proof of Theorem \ref{B3}. 
This first part is not only purely technical preparation for the second 
part of the paper. In fact, finding the image of the 
Burau representation seems to be a difficult problem, which is interesting 
by itself (see e.g. \cite{CL,CF,McM}).

\vspace{0.2cm}\noindent
In a sequel to the present article 
we will give some applications of these results 
to the study  of the images of the mapping class groups   
by quantum representations (see \cite{FK2}).

\vspace{0.2cm}\noindent 
{\bf Acknowledgements.}  We are grateful to 
Norbert A'Campo, J{\o}rgen Andersen,  Jean-Beno\^it Bost, Martin Deraux, 
Greg Kuperberg, Fran\c{c}ois Labourie, Yves Laszlo, Greg McShane, 
Ivan Marin, Gregor Masbaum, Daniel Matei, Majid Narimannejad, Christian Pauly, 
Bob Penner, Christophe Sorger and Richard Wentworth   
for useful discussions  and to the
referees for a careful reading of the paper leading to numerous 
corrections and  suggestions.
The first author was partially supported by 
the ANR Repsurf:ANR-06-BLAN-0311.
The second author is partially supported by Grant-in-Aid for Scientific
Research 20340010, Japan Society for Promotion of Science, and by World 
Premier International Research Center Initiative, MEXT, Japan.  
A part of this work was accomplished while the second
author was staying at Institut Fourier in Grenoble. He would like to thank
Institut Fourier for hospitality.

\section{Braid group representations}

\subsection{Jones and Burau's representations at roots of unity} 
In this section we recall the definition of the Jones 
and Burau's representations of the braid groups and 
show that they are equivalent except at 
primitive roots of unity of order 1 and 3. Moreover, we discuss 
when they are unitarizable or complex-unitarizable. 
We start with the following classical definition. 

\begin{definition}
The Temperley-Lieb algebra $A_{\tau,n}$, for $\tau\in\C^*$ and $n\geq 2$ 
is the $\C$-algebra generated by  the projectors $1,e_1,\ldots,e_{n-1}$ 
satisfying the relations:   
\[ e_j^2=e_j, \:\: j\in\{1,2,\ldots,n\},\]
\[ e_ie_j=e_je_i, \:\: {\rm if } \:\: |i-j|\geq 2, \]
\[ e_je_{j+1}e_j=e_je_{j-1}e_j=\tau e_j, \:\: j\in\{1,2,\ldots,n\}.\]
\end{definition}

\vspace{0.2cm}\noindent
There is a natural $\C^*$-algebra structure on $A_{\tau,n}$, obtained 
by setting $e_j^*=e_j$, $j\in\{1,2,\ldots,n\}$.  

\vspace{0.2cm}\noindent
According to Wenzl (\cite{We}) there exist such unitary projectors $e_j$, 
$1\leq j\leq n-1$, for any natural number $n\geq 2$ if and only if $\tau^{-1}\geq 4$ 
or $\tau^{-1}=4\cos^2\left(\frac{\pi}{k}\right)$, for some natural number 
$k\geq 3$. However, for given $n$ one could find projectors $e_1,\ldots,e_{n-1}$ 
as above if $\tau^{-1}=4\cos^2\left(\alpha\right)$, where 
the angle $\alpha$ belongs to some specific arc of the unit circle. 

\vspace{0.2cm}\noindent 
Another definition of the Temperley-Lieb algebra (which is equivalent 
to the former one, at least when  $\tau$ verifies the previous  conditions) is 
as a quotient of the Hecke algebra: 

\begin{definition}\label{tl}
The Temperley-Lieb algebra $A_{n}(q)$ is the quotient of the group 
algebra $\C B_n$ of the braid group $B_n$ 
by the relations:
\[ (g_i-q)(g_i+1)=0,\]
\[ 1+g_i+g_{i+1}+g_ig_{i+1}+g_{i+1}g_i+g_ig_{i+1}g_i=0,\]
where $g_i$ are the standard generators of the braid group $B_n$. 
The quotient obtained by imposing only the first relation above is 
called the Hecke algebra $H_n(q)$. 
\end{definition}

\vspace{0.2cm}\noindent
It is known that $A_n(q)$ is isomorphic to $A_{\tau,n}$ where 
$\tau^{-1}=2+q+q^{-1}$, and in particular, when $q$ is the 
root of unity $q=\exp\left(\frac{2\pi i}{n}\right)$. We suppose from now on 
that $\tau^{-1}=2+q+q^{-1}$. 

\vspace{0.2cm}\noindent
We will analyze  the case where $n=3$ and $q$ is a root of unity, and  
more generally for $|q|=1$. 
Then $A_{\tau,3}$ and $A_n(3)$ are nontrivial and well-defined 
for all $q$ with $|q|=1$ belonging to the arc of circle 
joining $\exp\left(-\frac{2\pi i}{3}\right)$ to 
$\exp\left(\frac{2\pi i}{3}\right)$. We will recover this result below in 
a slightly different context. 

\vspace{0.2cm}\noindent
Furthermore, $A_{\tau,3}$ is semi-simple and splits as 
$M_2(\C)\oplus \C$, where $M_2(\C)$ denotes the 
simple $\C$-algebra of 2-by-2 matrices. 
There is a natural representation of $B_3$  into $A_{\tau,3}$ 
which sends $g_i$ into $qe_i-(1-e_i)$. 
This representation is known to be unitarizable 
when $\tau^{-1}\geq 4$ (see \cite{Jones}).

\begin{proposition}\label{reducible}
Let $q=\exp(i \alpha)$. 
\begin{enumerate}
\item Assume that $q$ is not a primitive root of unity of order $2$ or $3$.   
Then  every completely reducible representation 
$\rho$ of $B_3$ into $GL(2,\C)$  which factors through 
$A_3(q)$ is equivalent to some representation 
$\rho_{q,C}$ defined by:   
\[ \rho_{q,C}(g_1)=\left(\begin{array}{cc}
q & 0 \\
0 & -1 \\
\end{array}
\right), \,\,  \rho_{q,C}(g_2)=\left(\begin{array}{cc}
-\frac{1}{q+1} & -(q+1)C \\
-\epsilon_q(q+1)\overline{C}r^2 & \frac{q^2}{q+1} \\
\end{array}
\right),\]
where $C\in \C-\{0\}$,  
$r^2=r(q,C)^2=|C|^{-2}|q+1|^{-4}|q+\overline{q}+1|$ and 
$\epsilon_q$ is the sign of the 
real number $q+\overline{q}+1$, namely:  
\[ \epsilon_q=\left\{\begin{array}{ll}
1, & {\rm if } \: \alpha\in (-\frac{2\pi}{3}, \frac{2\pi}{3});\\
-1, & {\rm if } \:\alpha\in (\frac{2\pi}{3},\pi)\cup(\pi,\frac{4\pi}{3}).\\
\end{array}\right.
\]
\item Let  $q$ be a  primitive root of unity of order $2$ or $3$.  
Then completely reducible representations 
$\rho$ of $B_3$ into $GL(2,\C)$  which factor through 
$A_3(q)$ are abelian with finite image and equivalent to: 
\[ \rho_{q,0}(g_1)=\rho_{q,0}(g_2)=\left(\begin{array}{cc}
q & 0 \\
0 & -1 \\
\end{array}
\right).
\]
We may extend the definition of $\epsilon_q, r(q,C)$ to this exceptional case 
by setting $\epsilon_q=1$, if $\alpha\in \left\{-\frac{2\pi}{3},\frac{2\pi}{3}\right\}$, $\epsilon_q=-1$, if $\alpha=\pi$ and $r(q,0)^2=1$. 
In this case $\rho_{q,0}$ is both unitarizable and complex-unitarizable.   
\item 
If $\alpha\in \left(-\frac{2\pi}{3}, \frac{2\pi}{3}\right)$, then 
the representation $\rho_{q,C}$ is unitarizable if $r(q,C)^2=1$. 
\item If $\alpha\in\left(\frac{2\pi}{3}, \frac{4\pi}{3}\right)$, then  
the representation $\rho_{q,C}$ is complex-unitarizable  if $r(q,C)^2=1$. 
\end{enumerate}
\end{proposition}
\begin{proof}
We can choose  
$\rho(g_1)=\left(\begin{array}{cc}
q & 0 \\
0 & -1 \\
\end{array}
\right)$ since completely reducible 2-dimensional  representations are 
diagonalizable and the eigenvalues are prescribed. 
Since $g_2$ is conjugate to $g_1$ in $B_3$  
we have $\rho(g_2)=U\rho(g_1)U^{-1}$, where, without loss 
of generality, we can suppose that $U\in SL(2,\C)$. 
We discard the case $q=-1$ from now on  
when the representation should be abelian, as $\rho(g_1)$ is scalar. 

\vspace{0.2cm}\noindent
Set $U=\left(\begin{array}{cc}
a & b \\
c & d \\
\end{array}
\right)$, where $ad-bc=1$. Then we have
 $\rho(g_2)=\left(\begin{array}{cc}
qad+bc & -(q+1)ab \\
(q+1)cd & -qbc-ad \\
\end{array}
\right)$. Therefore $\rho$ factors through $A_3(q)$, namely the second identity 
of Definition \ref{tl} is satisfied, if and only if:
\[ q ad +bc =-\frac{1}{q+1}. \]
If $1+q+q^2\neq 0$, we obtain the solutions: 
$d=\frac{q}{(q+1)^2a}$, 
and $c=-\frac{q^2+q+1}{(q+1)^2b}$. This implies that:      
\[\rho(g_2)=\left(\begin{array}{cc}
-\frac{1}{q+1} & -(q+1)C \\
-\frac{(q^2+q+1)q}{(q+1)^3C} & \frac{q^2}{q+1} \\
\end{array}
\right),\] 
which coincides with the matrix $\rho_{q,C}(g_2)$ in the statement 
of Proposition \ref{reducible},  where 
$C=ab$  and $r^2=r(q,C)^2=\frac{|q+\overline{q}+1|}{|q+1|^4|C|^2}$.

\vspace{0.2cm}\noindent
If $q$ is a primitive root of unity  
of order 3,  then we find $d=\frac{q}{(q+1)^2a}$ 
and either $b=0$ and $c$ arbitrary or $c=0$ and $b$ arbitrary. 
But the representation  $\rho$  is completely reducible only when  
$b=c=0$ and this gives the second claim of the Proposition \ref{reducible}. 

\vspace{0.2cm}\noindent
We re-scale the representation $\rho_{q,C}$ so that it takes values in 
$SL(2,C)$. This amounts to replace $\rho_{q,C}(g_j)$ by 
$\tilde{\rho}_{q,C}(g_j)=\lambda \rho(g_j)$, where $\lambda$ 
satisfies $\lambda^2q=-1$. 
Then the condition $r^2=1$ is equivalent to 
$\tilde\rho_{q,C}(g_2)=\left(\begin{array}{cc}
u & v \\
-\epsilon\overline{v} & \overline{u} \\
\end{array}
\right)$, where $|u|^2+\epsilon|v|^2=1$. 
In this case the representation $\tilde\rho_{q,C}$  
takes values in $U(2)$, when $\epsilon=1$ and 
in $U(1,1)$, when $\epsilon=-1$ respectively. 
\end{proof}

\begin{remark}
Notice that representations associated to the same 
$q,|C|^2$ are pairwise conjugate.   
\end{remark}

\vspace{0.2cm}\noindent
The representation $\rho_{q,C}$ of $B_3$ that arises 
as above and for which the parameter $C$ satisfies $r(q,C)^2=1$ will be called 
the {\em Jones representations} of $B_3$ at $q$. By the previous remark 
the conjugacy class of $\rho_{q,C}$ is uniquely determined  by the value of $q$.  
We omit the subscript $C$ in the sequel  when the choice of 
$C$ is not relevant. 

\begin{proposition}
Let $\tilde\rho:B_3\to SU(2)$ be a 
unitary Jones representation at $q=\exp(i\alpha)$, 
for $\alpha\in (-\frac{2\pi}{3}, \frac{2\pi}{3})$. Let $Q:SU(2)\to SO(3)$ 
be the standard double covering map. Then 
$Q\circ\tilde\rho(g_1)$ and $Q\circ\tilde\rho(g_2)$ are 
two rotations of angle $\pi+\alpha$, whose axes form an angle $\theta$ 
which is given by the formula:  
\[ \cos\theta= \frac{\cos \alpha}{1+\cos\alpha} \cdot \] 
\end{proposition}
\begin{proof}
The set of anti-Hermitian 2-by-2 matrices, namely  the matrices 
$A=\left(\begin{array}{cc}
w+ix & y+iz \\
-y +iz  & w-ix \\
\end{array}
\right)$ with real $w,x,y,z$,  is identified with the space 
${\mathbb H}$ of quaternions $w+ix+jy+kz$. Under this identification $SU(2)$ 
corresponds to the sphere consisting of the unit quaternions. 
In particular, any element  of $SU(2)$ acts by conjugacy on $\mathbb H$. 
Let $\R^3\subset \mathbb H$ be the vector subspace given  
by  the equation $w=0$. Then $\R^3$ is $SU(2)$-conjugacy invariant and the 
linear transformation induced by $A\in SU(2)$ on $\R^3$ is 
the orthogonal matrix $Q(A)\in SO(3)$. 

\vspace{0.2cm}\noindent
A simple computation yields the following explicit formula for $Q$: 
\[ Q\left(\begin{array}{cc}
w+ix & y+iz \\
-y +iz  & w-ix \\
\end{array}
\right)= 
\left(\begin{array}{ccc}
1-2(y^2+z^2) & 2(xy-wz) & 2(xz+wy)\\
2(xy+wz) &1-2(x^2+z^2)  & 2(yz-wx) \\
2(xz-wy) & 2(yz+wx) & 1-2(x^2+y^2) \\
\end{array}
\right).
\]

\vspace{0.2cm}\noindent
This shows that $Q(\tilde\rho(g_1))$ is the rotation of angle $\pi+\alpha$ 
around the axis $i\in \R^3$ in the space of imaginary quaternions. 
Instead of  giving the cumbersome computation of $Q(\tilde\rho(g_2))$ 
observe that $Q(\tilde\rho(g_2))$ is also a rotation of angle $\pi+\alpha$ 
since it is conjugate to $Q(\sigma_1)$. Let $N$ be the rotation
axis of $Q(\tilde\rho(g_2))$. If $N=ui +v j + wk$,   then $\cos\theta=u$. 

\vspace{0.2cm}\noindent
A direct computation shows that the matrix of the rotation of angle 
$\pi+\alpha$ around the axis $N$ has a matrix whose first entry 
on the diagonal reads $u^2+(1-u^2)\cos(\pi+\alpha)$. Therefore 
we have the identity: 
\[ u^2+(1-u^2)\cos(\pi+\alpha)=1-2(y^2+z^2),\]
where $y,z$ are the off diagonal entries of 
$\tilde\rho(g_2)$, namely:   
\[ y^2+z^2=|-\lambda(q+1)C|^2=\frac{|q+1+\overline{q}| }{|q+1|^2}.\]
This gives $|u|=\left|\frac{\cos \alpha}{1+\cos\alpha}\right|$. 
Identifying one more term in the matrix of $Q(\sigma_2)$ yields 
the sign of $u$. We omit the details. 
\end{proof}

\begin{remark}
In \cite{RS} the authors consider the structure of groups generated by 
two rotations of finite order for which axes form an angle which is 
an integral part of $\pi$. Their result is that  there are only few new 
relations. However, the previous Proposition shows that we cannot apply 
these results to our situation. It seems quite hard just to find 
those $\alpha$ for which the axes verify the condition from \cite{RS}. 
\end{remark}

\begin{definition}
The (reduced) Burau representation $\beta:B_n\to GL(n-1,\Z[q,q^{-1}])$ 
is defined on the standard generators 
\[ \beta_q(g_1)=\left(\begin{array}{cc}
-q & 1 \\
0  & 1 \\
\end{array}
\right) \oplus {\mathbf 1}_{n-3},\]
\[ \beta_q(g_j)={\mathbf 1}_{j-2}\oplus 
\left(\begin{array}{ccc}
1 & 0 & 0 \\
q & -q & 1 \\
0 & 0  & 1 \\
\end{array}
\right) \oplus {\mathbf 1}_{n-j-2}, \:\: {\rm for} \:\: 2\leq j\leq n-2,\]
\[ \beta_q(g_{n-1})={\mathbf 1}_{n-3}\oplus 
\left(\begin{array}{cc}
1 & 0 \\
q & -q \\
\end{array}
\right). 
\]
\end{definition}

\vspace{0.2cm}\noindent
Jones already observed in \cite{Jones} 
that the following  holds true for  the principal 
roots of unity, i.e., for the roots of unity of the form 
$\exp\left(\frac{2\pi i}{n}\right)$, $n\in\Z$: 

\begin{proposition}\label{jonesburau}
Burau's representation of $B_3$ at $q$  
is conjugate to the tensor product of the parity representation and the 
Jones representation at $q$, for all $q$ which are not primitive roots of unity of order $2$ or $3$. 
\end{proposition}
\begin{proof}
Recall that the parity representation $\sigma:B_3\to \{-1,1\}\subset 
\C^*$ is given by $\sigma(g_j)=-1$. 
Burau's representation for $n=3$ is given  by 
\[ \beta_q(g_1)=\left(\begin{array}{cc}
-q & 1 \\
0  & 1 \\
\end{array}
\right), \,\, 
\beta_q(g_2)=  
\left(\begin{array}{cc}
1 & 0 \\
q & -q \\
\end{array}
\right).\]
Take then $V=\left(\begin{array}{cc}
a & \frac{1}{(q+1)a} \\
0  & \frac{1}{a} \\
\end{array}
\right)$, for $q\neq -1$, where 
$a$ is given by 
%$\epsilon (q+1)\overline{C}= qa^2$. 
$(q+1)^3C a^2=1+q+q^2$ and $C\neq 0$ is chosen such that 
$\rho_{q,C}$ is unitarizable, namely 
$|C|^2=|q+1|^{-4}|1+q+\overline{q}|$. 
One verifies easily that 
$(\sigma\otimes \rho_{q,C})(g_j)=V^{-1}\beta_q(g_j)V$. 
\end{proof}

\begin{remark}
The definition of $A_{\tau,3}$ in terms of orthogonal projections 
has a unitary flavor and thus it works properly only when 
Burau's representation is unitarizable, namely 
only for those $q=\exp(i \alpha)$, where 
$\alpha\in  (-\frac{2\pi}{3}, \frac{2\pi}{3})$. 
\end{remark}

\subsection{Two conjectures of Squier and proof of Theorem \ref{intersection}}
This section is devoted to the study of the kernels of the 
Jones and Burau's representations at roots of unity. Our motivation 
comes from the following conjectures of Squier in \cite{Sq}: 
\begin{conjecture}[Squier]\label{unu}
The kernel of Burau's representation $\beta_{-q}$ for a primitive 
$k$-th root of unity $q$ is the normal subgroup $B_n[k]$ 
of $B_n$ generated by $g_j^k$, $1\leq j\leq n-1$. 
\end{conjecture}

\vspace{0.2cm}\noindent
The second conjecture of Squier, which is related to the former one, is: 

\begin{conjecture}[Squier]\label{doi}
The intersection of $B_n[k]$ over all $k$ is trivial. 
\end{conjecture}

\vspace{0.2cm}\noindent 
In order to prove Theorem \ref{intersection}, which  shows that 
a stronger version of   
Conjecture \ref{doi} holds 
we will first need a number of definitions and lemmas. 
Let $\Sigma_{0,n+1}$ be a disk with $n$ holes. The (pure)  mapping class 
group  $M(\Sigma_{0,n+1})$ is the group of framed pure braids 
$\widetilde{PB}_n$ and fits into the exact sequence: 
\[ 1\to \Z^{n}\to \widetilde{PB}_n\to PB_n\to 1\]
where $\Z^n$ is generated by the Dehn twists along the boundary 
curves. 

\vspace{0.2cm}\noindent
The extended mapping class group $M^*(\Sigma_{0,n+1})$ 
is the group of mapping classes of homeomorphisms 
of the disk with $n$ holes that fix point-wise the boundary of the disk 
but are allowed to permute the remaining boundary components, 
which are suitably parameterized. Thus $M^*(\Sigma_{0,n+1})$ is the group 
of framed braids on $n$ strands and we have then the  
exact sequence:  
\[ 1\to \Z^{n}\to M^*(\Sigma_{0,n+1})\to B_n\to 1.\]
Since the unit tangent bundle has a section the exact sequence above 
has a non-canonical splitting, i.e., 
there exists a section $s:B_n\hookrightarrow M^*(\Sigma_{0,n+1})$, 
which we fix once for all. The restriction of $s$ to the subgroup 
$PB_n$ yields a section $PB_n\hookrightarrow \widetilde{PB}_n$.
Let $g_1,\ldots,g_{n-1}$ denote the standard generators of $B_n$.

\begin{definition}
Let $k$ be a positive integer. 
The subgroup $B_n\{k\}$ of $B_n$ is the normal subgroup generated 
by the elements: 
\[ g_1^{2k}, (g_1g_2g_1)^{3k}, (g_1g_2g_3g_2g_1)^{4k}, 
\ldots, (g_1g_2\cdots g_{n-2}g_{n-1}g_{n-2}\cdots g_2g_1)^{nk}.\]
\end{definition}

\vspace{0.2cm}\noindent
Observe that $B_n[2k]\subset B_n\{k\}$. 

\begin{definition}
For any compact orientable surface $\Sigma$ (possibly with boundary) 
we set  $M(\Sigma)[k]$  for the normal subgroup 
of $M(\Sigma)$ generated by the $k$-th powers of Dehn twists.  
\end{definition}

\begin{lemma}\label{framed}
We have $s(B_n\{k\})\subset M(\Sigma_{0,n+1})[k]$. 
\end{lemma}
\begin{proof}
Every normal generator of $B_n\{k\}$  is a pure braid and hence 
$B_n\{k\}\subset PB_n$. Furthermore, let us observe that 
$\delta_j=(g_1g_2\cdots g_{j-2}g_{j-1}g_{j-2}\cdots g_2g_1)^{j}$ is a Dehn twist 
along a curve encircling  the first $j+1$ punctures 
of the  $n$-punctured disk.
Let  $\gamma$ be an embedded 
curve in the $n$-punctured disk which encircles $j+1$ punctures.  
Then the (right) Dehn twist $T_{\gamma}$ along the curve $\gamma$ is conjugate 
to $\delta_j$ by means of some homeomorphism 
of the $n$-punctured disk sending $\gamma$ into the the curve encircling 
the first $j+1$ punctures. 
Thus $B_n\{k\}$ is the group generated by  
the $k$-th powers of Dehn twists on the $n$-punctured disk. 

\vspace{0.2cm}\noindent
The lift of a Dehn twist $T_{\gamma}\in PB_n$  into the mapping class group 
$M(\Sigma_{0,n+1})$ of the $n$-holed disk is of the form 
$s(T_{\gamma})=T_{\gamma}\prod_{i}T_{c_i}^{\varepsilon_i}$, 
where $\varepsilon_i=\pm 1$ and 
$T_{c_i}$ are Dehn twists along boundary components of $\Sigma_{0,n+1}$.  
Therefore $s(T_{\gamma}^k)\in  M(\Sigma_{0,n+1})[k]$. 
This proves the claim.   
\end{proof}

\begin{remark}
In a similar way we can identify $B_n[2k]$ with 
the subgroup of $B_n$ generated by Dehn twists 
along curves encircling precisely $2$ punctures.  
\end{remark}

\vspace{0.2cm}\noindent
The main result of this section is the following one which implies 
immediately Theorem \ref{intersection} in the Introduction:

\begin{theorem}\label{intersection2}
The intersection of $B_n\{k\}$ over an infinite set of 
integers $k$ is trivial. 
\end{theorem}
\begin{proof}
When $n=2$, the claim holds trivially. Assume henceforth that $n\geq 3$. 
We embed $\Sigma_{0,n+1}$ into the closed orientable surface 
$\Sigma_{n+1}$ of genus $n+1$ by gluing a one-holed torus 
along each boundary component. 
Let $M_{n+1}$ denote the mapping class group of $\Sigma_{n+1}$.   
According to \cite{PR} the homomorphism 
$i: M(\Sigma_{0,n+1})\to M_{n+1}$ induced by the inclusion of surfaces 
is injective.

\vspace{0.2cm}\noindent
We have obviously $i(M(\Sigma_{0,n+1}))[k]\subset M_{n+1}[k]$ and 
so Lemma \ref{framed} implies that 
$i(s(B_n\{k\}))\subset M_{n+1}[k]$.

\vspace{0.2cm}\noindent
Recall that in \cite{BHMV} 
the authors defined the TQFT functor $\mathcal V_{p}$, for every $p\geq 3$  
and a primitive root of unity $A$ of order $2p$.
These TQFT should correspond to the so-called 
$SU(2)$-TQFT, for even $p$ and to 
the $SO(3)$-TQFT, for odd $p$ (see also \cite{LW} for another 
$SO(3)$-TQFT). 

\begin{definition}\label{qrep}
Let $p\in\Z_+$, $p\geq 3$, such that $p\not\equiv 2({\rm mod}\: 4)$. 
The  quantum representation $\rho_p$ 
is the projective representation of  the mapping class group 
associated to the TQFT $\mathcal V_{\frac{p}{2}}$ for even $p$
and $\mathcal V_{p}$ for odd $p$,   
corresponding to the following choices of the root of unity: 
\[ A_p=\left\{\begin{array}{ll}
-\exp\left(\frac{2\pi i}{p}\right), & {\rm if}\: p\equiv 0({\rm mod}\: 4);\\
-\exp\left(\frac{(p+1)\pi i}{p}\right) , & {\rm if}\: p\equiv 1({\rm mod}\: 2).\\
\end{array}\right. \]
\end{definition}

\vspace{0.2cm}\noindent
Notice that $A_p$ is a primitive root of unity of order $p$ when $p$ is even 
and of order $2p$ otherwise.  
Consider now the projective quantum representations $\rho_p$ of 
$M_{n+1}$ from Definition \ref{qrep}.  
According to \cite{A,FWW}, for any infinite set of 
even integers $A$ we have $\cap_{p\in A}\ker \rho_{p}=1$. 
However, the proof given in \cite{FWW} for the the $SU(2)$-TQFT 
extends without any essential modification 
to the $SO(3)$-TQFTs ${\mathcal V}_p$ defined in \cite{BHMV}.  
Therefore $\cap_{p\in A}\ker \rho_{p}=1$, for any infinite set 
of integers $A$. Recall that $\rho_p$ was defined in 
Definition \ref{qrep} only when $p\not\equiv 2({\rm mod}\: 4)$. 

\vspace{0.2cm}\noindent
The eigenvalues of a Dehn twist in the TQFT $\mathcal V_p$ i.e.,   
the entries of the diagonal $T$-matrix  
are of the form $\mu_l=(-A_p)^{l(l+2)}$, where $l$ belongs to the 
set of admissible colors (see \cite{BHMV}, 4.11). 
The set of admissible colors 
is $\{0,1,2,\ldots,\frac{p}{2}-2\}$, for even $p$ and is 
$\{0,2,4,\ldots, p-3\}$ for odd $p$.  
Therefore the order of the image of a Dehn twist by $\rho_p$ 
is $p$. 

\vspace{0.2cm}\noindent
Therefore, $M_{n+1}[p]\subset \ker \rho_p$, for any 
$p$. Then the previous results and the injectivity of $i\circ s$ imply that 
$\cap_{p\in A} B_n\{p\}=1$, for any infinite set $A$. 
\end{proof}

\begin{remark}
The weaker statement that $\cap_{k\in\Z-\{0\}}B_n[k]=1$ 
can also be shown by means of the 
residually finiteness of the braid group.
This was independently observed by Ivan Marin. 
Consider a residually finite group $G$ having  a finite 
system of generators $S$. 
Let $G[k]$ be the normal subgroup 
of $G$ generated by $s^k$, with $s\in S$. 
We claim that $\cap_{k\in\Z-\{0\}}G[k]=1$.
In fact, suppose that there exists 
$1\neq a\in \cap_{k\in\Z-\{0\}}G[k]$. 
By the residual finiteness of $G$ there exists 
some finite group $F$ and a morphism $f:G\to F$ with $f(a)\neq 1$. 
Now $f(s)^k=1$, for every $s\in S$, 
where $k$ is the order of the finite group $F$. This shows that 
$f$ factors through $G/G[k]$, which implies that $f(a)=1$, contradicting 
our assumption. This proves the claim.  
In particular, this implies that 
\[ \cap_{k\in\Z-\{0\}}B_n[k]= \cap_{k\in\Z-\{0\}}B_n\{k\}=1,\:\:\: \cap_{k\in\Z-\{0\}}M_n[k]=1.\]
However, it seems that  the proof of the stronger claim of 
Theorem \ref{intersection2} uses in an essential way the asymptotic 
faithfulness of the quantum representations.  
\end{remark}

\begin{proposition}\label{false}
Conjecture \ref{unu} is false for $n\geq 5$, for all but finitely many 
$q$ of even order. 
\end{proposition}  
\begin{proof}
One knows by results of Bigelow (\cite{Big}), Moody (\cite{Moody}), 
Long and Paton (\cite{LP})
that for $n\geq 5$ the (generic i.e., for a formal indeterminate $q$) 
Burau representation  $\beta$  into 
$GL(n-1,\Z[q,q^{-1}])$ is not faithful. Let $a\in B_n$ 
be such a non-trivial element in the kernel of $\beta$. 

\vspace{0.2cm}\noindent
Suppose that Conjecture \ref{unu} is true for infinitely 
many primitive roots of unity $q$ of even order.  
Then $a$ should belong to the intersection of kernels of 
all $\beta_q$, over all roots of unity $q$. 

\vspace{0.2cm}\noindent
By Theorem \ref{intersection} we have 
$\cap_{k=2}^{\infty} B_n[2k]=1$. 
If $\ker\beta_q=B_n[2k]$ for infinitely many roots of unity 
$q$ of even order $2k$, it follows  
that $a\in \cap_{k=2}^{\infty}B_n[2k]=1$, 
which is a contradiction. 
\end{proof}

\begin{remark}
The proof of the asymptotic faithfulness in \cite{FWW} is given for one 
primitive root of unity $q$ of given order. 
However, this proof works as well for any other 
primitive roots of unity, by using a Galois conjugacy. 
\end{remark}

\section{The image of Burau's representation of $B_3$ at roots of unity}

\subsection{Finite and exceptional quotients of $B_3$}
The aim of this section is to understand the image of  
Burau's representation $\beta_{-q}(B_3)$ at small roots of unity and,  
in particular, to find  an explicit presentation of it. 
Notice that we will consider the representation at the root 
$-q$, instead of $q$, for reasons that will appear later.

\vspace{0.2cm}\noindent
If one is interested to know whether 
$\beta_{-q}(B_3)$ is discrete one should first analyze 
the case when the image can be conjugated into $U(2)$, 
and then rescale it into  $SU(2)$.  
There the discreteness is equivalent to the finiteness of the image. 
The finiteness of the Jones representation of $B_3$ 
was completely characterized in \cite{Jones}.  
Jones studied the case where the  roots of unity $-q$
have the form $-q=\exp\left(\frac{2\pi i}{k}\right)$, but the 
Galois conjugation yields isomorphic groups so that the discussion 
in \cite{Jones} is complete.  The only cases where the image of 
the Jones representation of $B_3$ at $-q$ is finite 
is when $-q$ is a primitive root of unity of order $1,2,3,4,6$ or $10$.  
Moreover, Burau's representation is equivalent to the Jones representation 
only when the root of unity is neither $-1$ nor a primitive 
third root of unity. These excluded cases should be treated separately. 
For the sake of completeness we sketch the proofs below.

\begin{proposition}\label{finiteim}
Let $q$ be a primitive root of unity of order 
$n\in \{2,3,4,5\}$. Then $\beta_{-q}(B_3)$ is a finite group 
with the group presentation:
\[ \langle g_1,g_2 \:\: | \:\: g_1g_2g_1=g_2g_1g_2,\:\: g_1^n=g_2^n=1 \rangle. \]
\end{proposition}
\begin{proof}
Set $B_k(n)=B_k/B_k[n]$. Then Burau's representation $\beta_{-q}$ factors 
through $B_3(n)$ when $q$ is a primitive root of unity of order $n$.

\vspace{0.2cm}\noindent
Coxeter gave in \cite{cox1} the exhaustive list of the groups 
$B_k(n)$ which are finite, together with their respective description
(see also \cite{cox2,cox3}). The finite ones are those for which
$(k-2)(n-2) < 4$. Namely, when $k=3$, there is the following list:  
\begin{enumerate}
\item $B_3(2)$ is the symmetric group $S_3$; 
\item $B_3(3)$ is isomorphic to $SL(2, {\Z}/3\Z)$ 
(or the binary tetrahedral group $\Delta(2,3,3)$, see section \ref{tri} 
for definitions) and has order $24$;
\item $B_3(4)$ is isomorphic to the triangle group 
$\Delta(2,3,4)$ and has order $96$;
\item $B_3(5)$ is isomorphic to  ${GL}(2,{\Z}/5\Z)$ and 
has order $600$. 
  \end{enumerate}

\vspace{0.2cm}\noindent
Set $N(n)\subset B_3(n)$ for the group generated by the image of 
$(g_1g_2)^3$, which is a generator of the center of $B_3$. 
By a direct computation we show that $\beta_{-q}((g_1g_2)^3)=-q^3 \mathbf 1$ 
is a scalar matrix and thus $\beta_{-q}$ induces a well-defined 
homomorphism $\tilde\beta_{-q}: B_3(n)/N(n)\to PGL(2,\C)$. 
Furthermore, we have the following commutative diagram: 
\[\begin{array}{ccccccc}
1  \to  & N(n)& \to & B_3(n)&  \to &  B_3(n)/N(n)& \to 1 \\
        & \downarrow& & \beta_{-q}\downarrow & & \tilde\beta_{-q}\downarrow & \\
1  \to  & \C^*& \to & GL(2,\C)&  \to &  PGL(2,\C)& \to 1
\end{array}
\]
It follows that $\beta_{-q}((g_1g_2)^3)=-q^3 \mathbf 1$ 
has order $o(n)$, where 
$o(2)=1, o(3)=2, o(4)=4, o(5)=10$. Since the order of $N(n)$ is 
also $o(n)$ it follows that the restriction of $\beta_{-q}$ at $N(n)$ is 
injective. 

\vspace{0.2cm}\noindent
From the previously cited results of Coxeter we derive that: 
\[
B_3(n)/N(n)=\left\{\begin{array}{ll}
S_3, & {\rm if } \: n=2; \\
A_4, & {\rm if } \: n=3; \\
S_4, & {\rm if } \: n=4; \\
A_5, & {\rm if } \: n=5. \\ 
\end{array}\right.
\]
where $S_m$ and $A_m$ denote the symmetric and 
the alternating group on $m$ elements, respectively. 

\vspace{0.2cm}\noindent
A direct inspection shows that the image of $\tilde\beta_{-q}$ is neither  
trivial nor of order 2 and thus $\tilde\beta_{-q}$ should be injective since 
alternating groups $A_n$ are simple if $n\neq 4$. 
Alternatively, we can use directly the computations made by Jones 
in \cite{Jones}.
This implies that $\beta_{-q}$ is injective as well and, in particular, 
$\beta_{-q}(B_3)$ has the given presentation, establishing the claim. 
\end{proof}
 
\vspace{0.2cm}\noindent
The two excluded cases which have to be treated separately  
are as follows: 
\begin{proposition}\label{except}
\begin{enumerate}
\item If $q=1$, then $\beta_{-q}(B_3)$ is the subgroup $SL(2,\Z)$ of 
$GL(2,\C)$ with the presentation:
\[ \langle g_1,g_2 \:\: | \:\:  g_1g_2g_1=g_2g_1g_2, \: (g_1g_2)^6=1\rangle. \]
\item 
If $q$ is a primitive $6$-th root of unity, then the representation 
$\beta_{-q}$ of $B_3$ is not completely reducible and its image 
 $\beta_{-q}(B_3)$ has the presentation:
 \[ \langle g_1,g_2 \:\: | \:\:  g_1g_2g_1=g_2g_1g_2, \: g_1^6=1, \:
\: g_1^{-2}g_2=g_2g_1^{2} \rangle. \]
\end{enumerate}
\end{proposition}
\begin{proof}
The group $\beta_{-1}(B_3)$ is 
generated by the images of the generators, namely $\left(\begin{array}{cc}
1 & 1 \\
0  & 1 \\
\end{array}
\right)$ and 
$  
\left(\begin{array}{cc}
1 & 0 \\
-1 & 1 \\
\end{array}
\right)$, and thus it coincides with $SL(2,\Z)$ 
and the presentation follows. 

\vspace{0.2cm}\noindent
Let $q$ be a primitive $6$-th root of unity, so that   
$t=-q$ is a primitive third root of unity.  
Let $V=\left(\begin{array}{cc}
-t & 0 \\
1 & 1 \\
\end{array}\right)$. We denote by $\Gamma$ the subgroup 
$V^{-1}\rho_{t}(B_3)V$ of $GL(2,\C)$. 
Then the matrices $h_i=V^{-1}\beta_{t}(g_i)V$ are both 
upper triangular, namely: 
\[ h_1= \left(\begin{array}{cc}
1 & -t^2\\
0 & -t \\
\end{array}\right), \:\: 
h_2= \left(\begin{array}{cc}
1 & 0 \\
0 & -t \\
\end{array}\right).
\]
We have therefore:
\[ h_1h_2^{-1}= \left(\begin{array}{cc}
1 & t\\
0 & 1 \\
\end{array}\right), \:\: 
h_2^{-1}h_1= \left(\begin{array}{cc}
1 & t+1 \\
0 & 1 \\
\end{array}\right), \:\: 
h_2h_1^{-1}h_2^{-1}h_1= 
\left(\begin{array}{cc}
1 & 1 \\
0 & 1 \\
\end{array}\right).
\] 
Since the diagonal of the generators $h_i$ is $(1, -t)$ 
the group $\Gamma$ is contained in the group  
of matrices: 
\[ \tilde{\Gamma}=\left\{\left(\begin{array}{cc}
1 & r + st\\
0 & (-t)^m \\
\end{array}\right), \:\: m\in \Z/6\Z, r,s\in \Z\right\}\subset GL(2,\C).
\] 
Any matrix in $\tilde\Gamma$ can be written as a product  
\[h_2^m(h_1h_2^{-1})^s (h_2h_1^{-1}h_2^{-1}h_1)^r,\] 
such that $\Gamma$ coincides with $\tilde\Gamma$.  

\vspace{0.2cm}\noindent
Observe now that the map $p:\Gamma\to \Z/6\Z$ defined by: 
\[ p\left(\begin{array}{cc}
1 & r + st\\
0 & (-t)^m \\
\end{array}\right)=m \in \Z/6\Z
\]
is a well-defined homomorphism. Then we obtain the exact sequence:
\[ 1\to \Z^2\to \Gamma \to \Z/6\Z \]
where the inclusion $i:\Z^2\to \Gamma$ is given by 
$i(1,0)=h_1h_2^{-1}$ and 
$i(0,1)=h_2h_1^{-1}h_2^{-1}h_1$. Thus $\Gamma$ is a polycyclic group.
Denote by $u=h_1h_2^{-1}$ and $v=h_2h_1^{-1}h_2^{-1}h_1$ the two generators 
of the kernel of $p$.  
We obtain an explicit presentation of $\Gamma$ out of 
one of $\Z^2$ by  adding the generator $h_2$  of order $6$ 
whose image generates $p(\Gamma)$ and the  
relations which describe its action by conjugacy 
on $\Z^2$. Specifically, we have:
\[ \Gamma=\langle u,v, h_2 | uv=vu, h_2^6=1, 
h_2uh_2^{-1}=v^{-1}, h_2vh_2^{-1}=uv \rangle 
\]
Now, in order to  describe $\Gamma$ as a quotient of $B_3$ 
we add the redundant generator $h_1$ and  the braid 
relation and express $u,v$ in terms of the $h_i$.  The conjugacy relations 
are  now consequences of the braid relation while the 
commutativity relation is equivalent to 
$h_2h_1^2=h_1^{-2}h_2$. This gives the desired presentation for 
the image $\beta_{-q}(B_3)$. 
\end{proof}

\subsection{Discrete subgroups of $PU(1,1)$}\label{discret}
The aim of this section is to find whether the image of 
Burau'sepresentation $\beta_{-q}$ is a discrete subgroup in $PU(1,1)$. 
The main result of this section is the identification of 
the image of a free subgroup of $PB_3$ by Burau's representation 
with a group generated by two rotations. Then 
some results of Knapp, Mostow and Deraux (\cite{Deraux,Knapp,Mo}) give necessary and sufficient 
conditions for such a subgroup to be discrete. 

\vspace{0.2cm}\noindent
Let us denote by $A=\beta_{-q}(g_1^2)$ and $B=\beta_{-q}(g_2^2)$ and 
$C=\beta_{-q}((g_1g_2)^3)$. 
As is well-known $PB_3$ is isomorphic to the direct product 
${\mathbb F}_2\times \Z$, where ${\mathbb F}_2$ is freely generated by 
$g_1^2$ and $g_2^2$ and the factor $\Z$ is the center of $B_3$ 
generated by $(g_1g_2)^3$. 

\vspace{0.2cm}\noindent
It is simple to check that: 
\[ A=\left(\begin{array}{cc}
q^2 & 1+q \\
0  & 1 \\
\end{array}
\right), \,\, 
B=  
\left(\begin{array}{cc}
1 & 0 \\
-q-q^2 & q^2 \\
\end{array}
\right), \,\,
C=  
\left(\begin{array}{cc}
-q^3 & 0 \\
0 & -q^3 \\
\end{array}
\right).
\]

\vspace{0.2cm}\noindent
Recall that $PSL(2,\Z)$ is the quotient of $B_3$ by its center. 
Since $C$ is a scalar matrix  
the homomorphism $\beta_{-q}:B_3\to GL(2,\C)$ 
factors to a homomorphism  
$PSL(2,\Z)\to PGL(2,\C)$.

\vspace{0.2cm}\noindent
We will be concerned below with  the subgroup 
$\Gamma_{-q}$ of $PGL(2,\C)$ generated by  the images of 
$A$ and $B$ in  $PGL(2,\C)$.
When  $\beta_{-q}$ is unitarizable,  
the group $\Gamma_{-q}$ can be viewed as a subgroup of 
the complex-unitary  group $PU(1,1)$. 
Specifically,  consider the action of $\Gamma_{-q}$ on the 
projective line $\C P^1$. Let $V$ be the matrix in the proof of 
Proposition \ref{jonesburau}, namely:   
 $V=\left(\begin{array}{cc}
a & \frac{1}{(1-q)a} \\
0  & \frac{1}{a} \\
\end{array}
\right)$, for $-q\neq -1$, where 
$a$ is given by $(1-q)^3C a^2=1-q+q^2$ and $C\neq 0$ is chosen such that 
$\rho_{-q,C}$ is unitarizable. 
Denote the conjugate $V^{-1}ZV$ by $\overline{Z}$. 
We have then: 
\[ \overline{A}=\left(\begin{array}{cc}
q^2 & 0 \\
0  & 1 \\
\end{array}
\right), \, 
\overline{B}=\left(\begin{array}{cc}
\frac{1+q^2}{1-q} & \frac{1+q^3}{(1-q)^2a^2} \\
-q(1+q)a^2  & -\frac{q+q^3}{1-q} \\
\end{array}
\right), \,\, \overline{AB}=\left(\begin{array}{cc}
\frac{q^2-q^4}{1-q} & \frac{q^2+q^5}{(1-q)^2a^2} \\
-q(1+q)a^2  & -\frac{q+q^3}{1-q} \\
\end{array}
\right).
\]
since $AB=\left(\begin{array}{cc}
-q^3-q^2-q & q^2+q^3\\
-q-q^2  & q^2 \\
\end{array}
\right)
$.

\vspace{0.2cm}\noindent
We know that $V^{-1}\beta_{-q}V=\sigma\otimes \rho_{-q,C}$ 
and $\sigma\otimes \rho_{-q,C}$ is unitarizable simply 
by rescaling. In fact    
$\lambda (\sigma\otimes \rho_{-q,C})$ is complex-unitary 
(for those values of $-q$ considered in Proposition \ref{reducible}) when 
$\lambda$ verifies the condition $\lambda^2q=1$. 
Since scalar rescaling does not affect the 
class of the matrix in $PU(1,1)$ we can work directly with the 
classes of the matrices $\overline{A}$ and  
$\overline{B}$ in $PU(1,1)$.  

\begin{definition}
Let $q=\exp(i\alpha)$, with 
$\alpha\in \left(-\frac{\pi}{3},\frac{\pi}{3}\right)$. 
The group $\Gamma_{-q}\subset PU(1,1)$ is the subgroup 
generated by the classes $\beta_{-q}(g_1^2)$ and 
$\beta_{-q}(g_2^2)$, namely  the classes of matrices $\overline{A}, 
\overline{B}$ in $PU(1,1)$.
\end{definition}

\vspace{0.2cm}\noindent
It appears that the  search for discrete subgroups in the 
complex-unitary case is more interesting than in the unitary case 
since we can find infinite discrete subgroups 
of $PU(1,1)$. The main result  in this section is the following: 
\begin{proposition}\label{discrete}
Let $q=\exp(i\alpha)$, with $\alpha\in \left(-\frac{\pi}{3},\frac{\pi}{3}\right)$.  
Then the group $\Gamma_{-q}$ is a discrete subgroup 
of $PU(1,1)$ if and only if  $q=\exp\left(\frac{\pm 2\pi i}{n}\right)$, for  
$n\in \Z_+$ and $n\geq 7$. 
\end{proposition}
\begin{proof}
Recall that $PU(1,1)$ is a subgroup of $PGL(2,\C)$ which keeps 
invariant (and hence acts on) the unit disk ${\mathbb D}\subset \C P^1$. 
The action of $PU(1,1)$ on $\mathbb D$ is conjugate to the 
action of the isomorphic group $PSL(2,\R)$ on the upper half plane. 
The former is simply the action by isometries on the disk model of the 
hyperbolic plane. 

\vspace{0.2cm}\noindent
The key point of our argument is the existence of a fundamental domain 
for the action of $\Gamma_{-q}$ on $\mathbb D$. 
We will look to the fixed points of the isometries 
$\overline{A}, \overline{B}, \overline{AB}$ on the hyperbolic disk  
$\mathbb D$. 
We have the following list: 
\begin{enumerate}
\item $\overline{A}$ has  the fixed point set $\{0,\infty\}$ in $\C P^1$, and 
thus a unique fixed point in $\mathbb D$, namely its center $O$. 
\item $\overline{B}$ has the fixed point set  
$\left\{-\frac{1}{(1-q)a^2}, -\frac{q^2-q+1}{q(1-q)a^2}\right\}\subset \C P^1$ and 
thus a unique fixed point in $\mathbb D$, namely 
$P=-\frac{q^2-q+1}{q(1-q)a^2}$.
In fact, if $\cos(\alpha+\pi)\in[-\frac{1}{2}, -1]$, then 
\[ \left|\frac{q^2-q+1}{q(1-q)a^2}\right|=
|1-q||a|^2=\sqrt{|1-q+q^2|}=\sqrt{1+2\cos(\alpha+\pi)}\in [0,1].\] 
\item $\overline{AB}$ has  the fixed point set 
$\left\{-\frac{q}{(1-q)a^2}, -\frac{q^2-q+1}{(1-q)a^2}\right\}\subset \C P^1$ and 
thus a unique fixed point in $\mathbb D$, namely 
$Q=-\frac{q^2-q+1}{(1-q)a^2}$.
\end{enumerate}

\vspace{0.2cm}\noindent
We have now the following lemma, whose proof is postponed 
a few lines later: 
\begin{lemma}\label{fdtdom}
The elements $\overline{A}, \overline{B}$ and $\overline{AB}$ of 
$PU(1,1)$ are rotations of the same angle $2\alpha$ centered 
at the three vertices of the equilateral geodesic triangle  
$\Delta=OPQ$ in $\mathbb D$, whose angles are equal to $\alpha$. 
\end{lemma}
\vspace{0.2cm}\noindent
Eventually we state the following result of Knapp  from \cite{Knapp},  
later rediscovered by  Mostow (see \cite{Mo}) and Deraux 
(\cite{Deraux}, Theorem 7.1): 
\begin{lemma}
The  three rotations of angle $2\alpha$ 
in $\mathbb D$ around the vertices of an 
equilateral  hyperbolic triangle $\Delta$ of angles $\alpha >0$ 
generate a discrete subgroup of 
$PU(1,1)$ if and only if 
$\alpha=\frac{2\pi}{n}$, with $n\in \Z_+$ and $n\geq 7$.  
\end{lemma}

\vspace{0.2cm}\noindent
Notice that the existence of a 
hyperbolic triangle of angles equal to $\alpha$ 
requires that $n\geq 7$. 

\vspace{0.2cm}\noindent
The two lemmas from above yield the result claimed in 
Proposition \ref{discrete}. 
\end{proof}

\begin{proof}[Proof of Lemma \ref{fdtdom}]
We know from above that $\overline{A}, \overline{B}$ and $\overline{AB}$ are 
elliptic elements of $PU(1,1)$. Actually all of them are rotations of  
angle $\pm 2\alpha$: 
\begin{enumerate}
\item $\overline{A}(z)=q^2z$ and hence $\overline{A}$ is the 
counterclockwise rotation of angle 
$2\alpha$ around $O$;
\item $\overline{B}$ is conjugate to $\overline{A}$ and 
thus is a rotation of angle $\pm 2\alpha$   around $P$;
\item $\overline{AB}$ has the 
eigenvalues $-q^3$ and $-q$, which are 
distinct since $q^2\neq 1$,  and so is diagonalizable. Therefore  
$\overline{AB}$ is a rotation of 
angle $\pm 2\alpha$ around $Q$. 
\end{enumerate}

\vspace{0.2cm}\noindent
Consider now the geodesic triangle $\Delta=OPQ$ in $\mathbb D$. 
The angle $\widehat{POQ}$ at $O$ equals $\alpha$ since 
$Q=qP$. Since the argument of $q$ is acute it follows that 
the orientation of  the arc $PQ$ is counterclockwise. 
Moreover, this shows that  $d(O,P)=d(O,Q)$, 
where $d$ denotes the hyperbolic distance in $\mathbb D$ and  hence 
we obtain the equality of angles $\widehat{OPQ}=\widehat{OQP}$. 

\vspace{0.2cm}\noindent
Let us introduce the element $D=\beta_{-q}(g_2)$, which
verifies $D^2=B$. Then $\overline{D}^2=\overline{B}$. We can compute  
\[ \overline{D}=\left(\begin{array}{cc}
\frac{1}{1-q} & \frac{q^2-q+1}{(1-q)^2a^2} \\
-qa^2  & \frac{q^2}{1-q} \\
\end{array}
\right).
\]
We know that $\overline{D}$ 
is a rotation of angle $\pm \alpha$  around $P$ 
since is conjugated to $\beta_{-q}(g_1)$. 
We can check that $\overline{D}(Q)=0$ and hence $\overline{D}$ 
is the counterclockwise rotation of angle $\alpha$ around $P$ and 
$d(P,Q)=d(P,O)$. Thus all angles of the triangle $\Delta$ are equal to 
$\alpha$. This also shows that $\overline{B}$ is the counterclockwise 
rotation of angle $2\alpha$. 

\vspace{0.2cm}\noindent
Since both $\overline{A}$ and $\overline{B}$ are counterclockwise 
rotations of angle $2\alpha$ it follows that $\overline{AB}$ is 
also the counterclockwise rotation of angle $2\alpha$. 
\end{proof}

\subsection{Triangle groups as images of a free pure braid subgroup}\label{tri}
The aim of this section is to obtain finite presentations for the groups $\Gamma_{-q}$. 
Discrete subgroups of $PU(1,1)$ have explicit 
presentations by means of a fundamental 
domain for their action on the hyperbolic disk 
$\mathbb D$. This method leads us to an identification of $\Gamma_{-q}$ with a suitable 
triangle group.   

\vspace{0.2cm}\noindent
Before we proceed we make a short digression on triangle groups.
Let $\Delta$ be a geodesic triangle in the hyperbolic plane of angles 
$\frac{\pi}{m},\frac{\pi}{n},\frac{\pi}{p}$, so that  
$\frac{1}{m}+\frac{1}{n}+\frac{1}{p}<1$.  
The extended triangle group $\Delta^*(m,n,p)$ is the group of isometries 
of the hyperbolic plane generated by the three reflections 
$R_1,R_2,R_3$ with respect to the edges  of $\Delta$. It is well-known that 
a presentation of  $\Delta^*(m,n,p)$ is given by 
\[  \Delta^*(m,n,p)=\langle R_1,R_2,R_3\; ; \; 
R_1^2=R_2^2=R_3^2=1,\; (R_1R_2)^m=(R_2R_3)^n=(R_3R_1)^p=1\rangle. \]
The second type of relations have a simple geometric meaning. 
In fact, the product of the reflections with respect to two 
adjacent edges is a 
rotation by the angle which is twice the angle between those edges.
The subgroup $\Delta(m,n,p)$ 
generated by the rotations $a=R_1R_2$, $b=R_2R_3$, $c=R_3R_1$ 
is a normal subgroup of index 2, which coincides  
with the subgroup of isometries preserving the orientation. 
One calls $\Delta(m,n,p)$ the triangle (also called triangular, or 
von Dyck) group 
associated to $\Delta$.
Moreover, the triangle group has the 
presentation: 
\[ \Delta(m,n,p)=\langle a,b,c\; ; \; 
a^m=b^n=c^p=1, abc=1\rangle. \]
Observe that $\Delta(m,n,p)$ also makes sense  when $m,n$ or $p$ are 
negative integers, by interpreting the associated generators as 
clockwise rotations. 
The triangle $\Delta$ is a fundamental domain for the action 
of $\Delta^*(m,n,p)$ on the hyperbolic plane. Thus 
a fundamental domain for $\Delta(m,n,p)$ consists of the 
union $\Delta^*$ of $\Delta$ with the reflection of $\Delta$ in one of its 
edges.

\begin{proposition}\label{even}
Let $m<k$ be such that ${\rm gcd}(m,k)=1$ where $k\geq 4$. 
Then the group 
$\Gamma_{-\exp\left(\frac{\pm 2m\pi i}{2k}\right)}$ is a triangle 
group with  the presentation:  
\[ 
\Gamma_{-\exp\left(\frac{\pm 2m\pi i}{2k}\right)}=\langle A,B; A^k=B^k=(AB)^k=1\rangle. \]
\end{proposition}
\begin{proof}
Denote by $\Delta(\frac{\pi}{\alpha},\frac{\pi}{\alpha},\frac{\pi}{\alpha})$
the group generated by the rotations of angle $2\alpha$ around 
vertices of the triangle $\Delta$ of angles $\alpha$. 
We will use this notation even when $\alpha$ 
is not an integral part of $\pi$ i.e., $\alpha$ cannot be written as 
$\frac{\pi}{k}$, with $k\in\Z$. 
We saw above that  $\Gamma_{-q}$ is isomorphic to 
$\Delta(\frac{\pi}{\alpha},\frac{\pi}{\alpha},\frac{\pi}{\alpha})$. 

\vspace{0.2cm}\noindent
When $\alpha=\frac{2\pi}{2k}$, the group 
$\Delta(\frac{\pi}{\alpha},\frac{\pi}{\alpha},\frac{\pi}{\alpha})$   
is a triangle group, namely it has the rhombus
$\Delta^*$ as a fundamental domain for its action on $\mathbb D$. 
In particular, $\Gamma_{-q}$ is the triangle 
group with the given presentation. 

\vspace{0.2cm}\noindent
For the general case of $\alpha=\frac{2\pi m}{2k}$ where 
$q$ is a primitive $2k$-th root of unity the situation is however quite 
similar. 
There is a Galois conjugation sending $-q$ into 
$-\exp\left(\frac{\pm 2\pi i}{2k}\right)$, which induces 
an automorphism of $PGL(2,\C)$. Although this automorphism does 
not preserve the discreteness it is an isomorphism of 
$\Gamma_{-q}$ onto  $\Gamma_{-\exp\left(\frac{\pm 2\pi i}{2k}\right)}$. 
This settles the claim.  
\end{proof}

\vspace{0.2cm}\noindent
If $n$ is odd $n=2k+1$, then the group $\Gamma_{-q}$ is a quotient of the 
triangle group associated to $\Delta$, which embeds into the group 
associated to some sub-triangle $\Delta'$ of $\Delta$. 

\begin{proposition}\label{odd}
Let $0<m<2k+1$ be such that ${\rm gcd}(m,2k+1)=1$ and $k\geq 3$. 
Then the group 
$\Gamma_{-\exp\left(\frac{\pm 2m\pi i}{2k+1}\right)}$ is  isomorphic 
to the  triangle group  $\Delta(2,3,2k+1)$ and has the 
following presentation (in terms of our generators $A,B$): 
\[ 
\Gamma_{-\exp\left(\frac{\pm 2m\pi i}{2k+1}\right)}=
\langle A,B; A^{2k+1}=B^{2k+1}=(AB)^{2k+1}=1, \,
(A^{-1}B^k)^2=1, \, (B^kA^{k-1})^3=1 \rangle. \]
\end{proposition}
\begin{proof} 
It suffices to consider the case $m=1$, 
as in the previous Proposition. 
The proof of the discreteness in (\cite{Deraux}, Theorem 7.1) 
shows that the group  $\Delta(\frac{2k+1}{2},\frac{2k+1}{2},\frac{2k+1}{2})$, 
which is generated by the rotations $a,b,c$ around the vertices  of 
the triangle $\Delta$   
embeds into the triangle group associated to a smaller triangle 
$\Delta'$. One constructs $\Delta'$ by considering all 
geodesics of $\Delta$ joining a vertex and the midpoint of its 
opposite side. The three median geodesics pass through the 
barycenter of $\Delta$ and subdivide $\Delta$ into 6 equal triangles.
We can take for $\Delta'$ any one of the 6 triangles of the subdivision. 
It is immediate that $\Delta'$ has angles 
$\frac{\pi}{2k+1}, \frac{\pi}{2}$ and $\frac{\pi}{3}$ so that the associated 
triangle group is $\Delta(2,3,2k+1)$. 
This group has the presentation:  
\[  \Delta(2,3,2k+1)=\langle \alpha, u, v \, ; \, \alpha^{2k+1}=u^3=v^2=\alpha uv=1\rangle, \]
where the generators are the rotations of double angle around the 
vertices of the triangle $\Delta'$.

\begin{lemma}\label{oddlem}
The natural embedding of $\Delta(\frac{2k+1}{2},
\frac{2k+1}{2},\frac{2k+1}{2})$ into 
$\Delta(2,3,2k+1)$ is an isomorphism. 
\end{lemma}
\begin{proof}
A simple geometric computation shows that:  
\[ a= \alpha^2, \, b=v \alpha^2v=u^2\alpha^2u, \, 
c=u\alpha^2 u^2. \]
Therefore $\alpha=a^{k+1}\in \Delta(\frac{2k+1}{2},
\frac{2k+1}{2},\frac{2k+1}{2})$.

\vspace{0.2cm}\noindent
From the relation 
$\alpha u v=1$ we derive $a^{k+1}uv=1$, and 
thus $u=a^kv$. The relation $u^3=1$ reads now 
$a^k(va^kv)a^kv=1$ and replacing $b^k$ by $va^kv$ we find that  
$v=a^kb^ka^k\in \Delta(\frac{2k+1}{2},
\frac{2k+1}{2},\frac{2k+1}{2})$. 

\vspace{0.2cm}\noindent
Further $u=a^kv=a^{-1}b^ka^k\in \Delta(\frac{2k+1}{2},
\frac{2k+1}{2},\frac{2k+1}{2})$. This means that 
$\Delta(\frac{2k+1}{2},
\frac{2k+1}{2},\frac{2k+1}{2})$ is actually  
$\Delta(2,3,2k+1)$, as claimed.
\end{proof}

\vspace{0.2cm}\noindent
It suffices now to find a presentation of $\Delta(2,3,2k+1)$ that uses the 
generators $A=a, B=b$. It is not difficult to 
show that the group  with the presentation of the statement 
is isomorphic to $\Delta(2,3,2k+1)$, the inverse 
homomorphism sending $\alpha$ into $A^{k+1}$, 
$u$ into $A^{-1}B^kA^k$ and $v$ into $A^kB^kA^k$.  
\end{proof}

\vspace{0.2cm}\noindent 
A direct consequence of Propositions \ref{even} and \ref{odd} is the 
following abstract description of the image of Burau's representation: 

\begin{corollary}\label{inftriang}
If  $q$ is not a primitive root of unity of order in   
the set $\{1,2,3,4,6,10\}$,
then $\Gamma_q$ is an infinite triangle group. 
\end{corollary}

\vspace{0.2cm}\noindent 
Alternatively, we obtain a set of normal generators for the kernel 
of Burau's representation, as follows: 

\begin{corollary}\label{kernelburau}
Let $n\not\in\{1, 6\}$ and $q$ a primitive root of unity of order $n$. 
We denote by $N(G)$ the normal closure of a subgroup 
$G$ of $\langle g_1^2,g_2^2\rangle$. 
Then the kernel 
$\ker \beta_{-q}:{\langle g_1^2,g_2^2\rangle\to PGL(2,\C)}$
of the restriction of Burau's representation is given by: 
\[
\left\{\begin{array}{ll} 
N(\langle g_1^{2k},g_2^{2k}, (g_1^2g_2^2)^k\rangle), & {\rm if }\: n=2k;\\
N(\langle g_1^{2(2k+1)},g_2^{2(2k+1)}, (g_1^2g_2^2)^{2k+1},  
(g_1^{-2}g_2^{2k})^{2}, (g_2^{2k}g_1^{2(k-1)})^3\rangle), & {\rm if }\: n=2k+1.\\ 
\end{array}\right.\]  
\end{corollary}

\subsection{Proof of Theorem \ref{B3}}
In order to prove Theorem \ref{B3} we need some preliminary lemmas 
explaining how to retrieve the kernel of Burau's representation 
of $B_3$ from known information on its restriction to the free 
subgroup $\langle g_1^2,g_2^2\rangle$ of $PB_3$. 

\vspace{0.2cm}\noindent 
The case when $q$ is of odd order is particularly simple: 
\begin{lemma}\label{oddl}
If $n=2k+1$, $k\geq 3$, the inclusion $PB_3\subset B_3$ induces an isomorphism: 
\[ \frac{PB_3}{PB_3\cap \ker\beta_{-q}}\to 
\frac{B_3}{\ker\beta_{-q}} \cdot\]
Equivalently, we have an exact sequence:  
\[ 1\to PB_3\cap \ker\beta_{-q}\to \ker\beta_{-q}\to S_3 \to 1.\]
\end{lemma}
\begin{proof}
The induced map is clearly an injection. 
Observe next that $g_1^{2k+1}, g_2^{2k+1}\in \ker \beta_{-q}$ 
and thus for every $x\in B_3$ there exists some 
$\eta\in \ker\beta_{-q}$ such that $\eta x\in PB_3$. 
Thus the image of the class $\eta x$ is  the class of $x$ and this 
shows that the induced homomorphism is also surjective. 
The claims follow. 
\end{proof}

\vspace{0.2cm}\noindent 
When $q$ has an even order we will need an additional combinatorial argument:
\begin{lemma}\label{evenl}
If $n=2k$, $k\geq 4$, then $\ker\beta_{-q}\subset PB_3$. 
Thus the inclusion $PB_3\subset B_3$ induces the exact sequence:  
\[ 1\to \frac{PB_3}{PB_3\cap \ker\beta_{-q}}\to 
\frac{B_3}{\ker\beta_{-q}}\to S_3\to 1\]
\end{lemma}
\begin{proof}
It suffices to show that $\beta_{-q}(g)\not\in \beta_{-q}(PB_3)$ 
for $g\in \{g_1,g_2,g_1g_2,g_2g_1,g_1g_2g_1\}$. 
Since none of $\beta_{-q}(g)$, for $g$ as above is a scalar matrix, this  
claim is equivalent to show that 
  $\beta_{-q}(g)\not\in \beta_{-q}(\langle g_1^2,g_2^2\rangle)=
\langle A,B \rangle$. 
We will  conjugate everything and work instead with 
$\overline{A}$ and $\overline{B}$. The triangle group generated by 
$\overline{A}$ and $\overline{B}$ has a fundamental domain 
consisting of the rhombus $\Delta^*$, which is  the union of 
$\Delta$ with its reflection image $R_j\Delta$. 
The common edge of the two triangles 
of the rhombus will be called a diagonal.

\vspace{0.2cm}\noindent
The image of $g_i$ is the rotation of angle $\alpha$ around a vertex 
of the triangle $\Delta$. If this rotation were an element of  
$\Delta(k,k,k)$, then it would 
act as an automorphism of the tessellation with copies of $\Delta^*$.  
When the vertex fixed by $g_i$ lies on the diagonal of $\Delta^*$, 
then a rotation of angle $\alpha$ sends the rhombus onto an 
overlapping rhombus (having one triangle in common) and thus it cannot be an automorphism of the tessellation, which is a contradiction.

\vspace{0.2cm}\noindent
This argument does not work when the vertex is 
opposite to the diagonal. However, let us color the triangle $\Delta$ 
in white and $R_j\Delta$ in black. Continue this way by coloring all 
triangles in black and white so that adjacent triangles have 
different colors. It is easy to see that the rotations 
of angle $2\alpha$ (and hence all elements of the group  $\Delta(k,k,k)$) 
send white triangles into white triangles. But the rotation of angle 
$\alpha$ around a vertex opposite to the diagonal sends a white 
triangle into a black one. This contradiction shows that the image of the $g_i$ 
does not belong to $\Delta(k,k,k)$. 

\vspace{0.2cm}\noindent
The last cases are quite similar. The images of $g_1g_2$ and $g_2g_1$ 
send $\Delta^*$ into an overlapping rhombus having one triangle in common. 
Eventually the image of $g_1g_2g_1$ does not preserve the black and white 
coloring. This proves the lemma. 
\end{proof}

\vspace{0.2cm}\noindent
We are now able to prove 
Theorem \ref{B3}, which we restate here for the reader's convenience: 
\begin{theorem}\label{B3re}
Assume that $q$ is a primitive $n$-th root of unity and $g_1,g_2$ are the 
standard generators of $B_3$. Then 
$\beta_{-q}(B_3)$ has a presentation with generators 
$g_1,g_2$ and relations: 
\begin{enumerate}
\item The case $n=2k$ and $k$ is odd: 
\[\begin{array}{lll}
\mbox{\rm Braid relation:} &  &  g_1g_2g_1=g_2g_1g_2, \\
\mbox{\rm Power relations:} &  & g_1^{2k}=g_2^{2k}=(g_1^2g_2^2)^k=1. \\
\end{array}\]
\item The case $n= 2k$ and $k$ is even: 
\[\begin{array}{lll}
\mbox{\rm Braid relation:} &  &  g_1g_2g_1=g_2g_1g_2, \\
\mbox{\rm Power relations:} &  & g_1^{2k}=g_2^{2k}=(g_1^2g_2^2)^{2k}=1. \\ 
\end{array}\]
\item The case $n= 2k+1$: 
\[\begin{array}{lll}
\mbox{\rm Braid relation:} &  &   g_1g_2g_1=g_2g_1g_2, \\
\mbox{\rm Power relations:} &  &  g_1^{2k+1}=g_2^{2k+1}=(g_1^2g_2^2)^{2(2k+1)}=1.\\
\end{array}\]
\end{enumerate}
\end{theorem}
\begin{proof}
When $n\in \{2,3,4,5\}$, this is already proved in Proposition \ref{finiteim}. 
We suppose then $n\geq 7$. 

\vspace{0.2cm}\noindent 
The strategy of the proof is to lift the triangle group presentation of 
$\Gamma_{-q}$ to $\beta_{-q}(\langle g_1^2,g_2^2\rangle)$ 
and then to $\beta_{-q}(PB_3)$, by adding 
a central generator. We add further the standard 
generators $g_1,g_2$ of $B_3$ and use the previous two lemmas 
in order to obtain a presentation of $\beta_{-q}(B_3)$ and then get rid 
of redundant generators and relations.

\vspace{0.2cm}\noindent 
Lemma \ref{evenl} shows that  $\ker \beta_{-q}$ 
has the same normal generators as 
$\ker\beta_{-q}\cap PB_3$, when $n$ is even. Lemma \ref{oddl} states 
that for odd $n=2k+1$ a set of normal generators of $\ker \beta_{-q}$
is obtained by adding the two elements 
$g_1^{2k+1}$ and $g_2^{2k+1}$ to a set of 
normal generators of $\ker\beta_{-q}\cap PB_3$.
In this way one produces a presentation of $\beta_{-q}(B_3)$ from a 
presentation of $\beta_{-q}(PB_3)$. 

\vspace{0.2cm}\noindent 
Furthermore, $PB_3$ is the direct product of the free group 
$\langle g_1^2,g_2^2\rangle$ with the center of $B_3$, which is generated 
by $(g_1g_2)^3$. 
Now $\beta_{-q}(g_1g_2)^3$ is the scalar matrix $-q^3{\mathbf 1}$. 
The order of $-q^3$ is $6k/({\rm gcd}(3,k){\rm gcd}(2,k+1))$ if 
$q$ is a primitive $2k$-th root of unity and is equal to 
$r=6(2k+1)/{\rm gcd}(3,2k+1)$ when $q$ is a primitive $2k+1$-th root of unity. 
Therefore a presentation of $\beta_{-q}(PB_3)$ can be obtained 
from a presentation of $\beta_{-q}(\langle g_1^2,g_2^2\rangle)$ 
by adjoining  a new central generator $(g_1g_2)^3$ and 
the following center relations:
\[ (g_1g_2)^{6k/ {\rm gcd}(2,k+1){\rm gcd}(3,k)}=1, {\rm for \: even}\: n=2k, \]
\[ ( g_1g_2)^{6(2k+1)/{\rm gcd}(3,2k+1)}=1,  {\rm for \: odd}\: n=2k+1.\] 
This new central generator will be redundant as soon as we pass to $B_3$ 
with its standard generators $g_1,g_2$. 

\vspace{0.2cm}\noindent  
The group $\beta_{-q}(\langle g_1^2,g_2^2\rangle)\subset GL(2,\C)$ 
is a central extension of its image mod scalars 
$\Gamma_{-q}\subset PGL(2,\C)$. Thus 
we can obtain a presentation of it by looking at 
the lifts of the relations  holding in $\Gamma_{-q}$. 

\vspace{0.2cm}\noindent 
Let $n=2k$. The lifts of the relations $A^k=B^k=1$ in $\Gamma_{-q}$ 
are the relations $g_1^{2k}=g_2^{2k}=1$ in 
$\beta_{-q}(\langle g_1^2,g_2^2\rangle)$. The eigenvalues of 
the matrix $AB$  are $-q^3$ and $-q$ so that  
\[\beta_{-q}((g_1^2g_2^2)^k)= \left\{\begin{array}{ll}
-\mathbf 1, & {\rm if }\: k\equiv 0({\rm mod }\: 2); \\
\mathbf 1, & {\rm if }\: k\equiv 1({\rm mod }\: 2). \\
\end{array}\right.
\]
Thus for odd $k$ it is enough to add the relation 
$(g_1^2g_2^2)^k=1$. 

\vspace{0.2cm}\noindent 
For even $k$ the element $(g_1^2g_2^2)^k$ is central of order 2. 
On the other hand, one proves  by recurrence on $m$ that the following 
combined relation holds true in $B_3$:
\[ (g_1g_2)^{3m}=  g_1^{2m}g_2(g_1^2g_2^2)^mg_2^{-1}.\]
Taking $m=k$ and recalling that $(g_1g_2)^3$ is central we 
find that $g_1^{2k}=1$  implies that: 
\[ (g_1^2g_2^2)^k=(g_1g_2)^{3k}.\]
Thus the fact that $(g_1^2g_2^2)^k$ is central is a consequence 
of the braid and power relations. Thus it suffices to add 
the power relation $(g_1^2g_2^2)^{2k}=1$, in order to get a presentation 
of $\beta_{-q}(B_3)$.

\vspace{0.2cm}\noindent 
For odd $n=2k+1$ the lifts of the relations $A^n=B^n=1$ are 
$g_1^{2n}=g_2^{2n}=1$, which are consequences of the 
power relations $g_1^n=g_2^n=1$. 
Furthermore, we verify that:
\[\beta_{-q}((g_1^2g_2^2)^{2k+1})= -\mathbf 1,\]
hence $(g_1^2g_2^2)^{2k+1}$ is central of order 2. 
The argument used above for even $k$ 
shows that $g_1^{2k+1}=1$ and the braid relations imply that 
$(g_1^2g_2^2)^{2k+1}$ is central, so it suffices to add the 
last power relation $(g_1^2g_2^2)^{2(2k+1)}=1$. 
The remaining lifts of relations in $\Gamma_{-q}$ are redundant. 
In fact, braid and powers relations give us:  
\[ (g_1^{-2}g_2^{2k})^2=(g_1^{-2}g_2^{-1})^2=(g_1g_2)^{-3},\]
\[ (g_1^{2k}g_2^{2k-2})^3=(g_1^{-1}g_2^{-3})^3=(g_1g_2)^{-6}.\]

\vspace{0.2cm}\noindent 
Eventually, a direct inspection shows that center relations  
are obtained from the combined relation above along with the braid 
and power relations. 
\end{proof}

{
\small      
      
\bibliographystyle{plain}

}

\end{document}